\documentclass[11pt,letterpaper]{article}
\usepackage[margin=1.5in]{geometry}
\usepackage{amsmath,times}
\usepackage{amssymb,amsmath}
\usepackage{graphicx}
\usepackage{amsthm}
\usepackage{multicol}
\usepackage[all]{xy}
\usepackage{cancel}
\usepackage{tikz}
\usepackage[utf8]{inputenc}
\usepackage{url}

\title{Cardinalities in finite monoids of $G$-equivariant functions}
\author{Ram\'on H. Ruiz-Medina\footnote{Email: harath.ruiz@academicos.udg.mx} \\
\small{Centro Universitario de Ciencias Exactas e Ingenier\'ias}, \\ 
\small{Universidad de Guadalajara, Guadalajara, M\'exico.}}
\date{}

\newtheorem{teorema}{Theorem}[]
\newtheorem{lema}[teorema]{Lemma}

\newtheorem{ejemplo}[teorema]{Example}

\newcommand{\EndG}{\mathrm{End}_{G}(X)}
\newcommand{\AutG}{\mathrm{Aut}_{G}(X)}
\newcommand{\Hbox}{\mathcal{B}_{[H]}}
\newcommand{\Kbox}{\mathcal{B}_{[K]}}

\newcommand{\StabG}{\mathrm{Stab}_{G}(X)}
\newcommand{\ConjG}{\mathrm{Conj}(G)}
\newcommand{\ConjX}{\mathrm{Conj}_{G}(X)}

\begin{document}

\maketitle

\begin{abstract}
A set with a group action is referred to as a $G$-set, and the set of functions that commute with this action forms a monoid under function composition. This paper examines the case where the $G$-set is finite, which implies that the monoid of $G$-equivariant functions is also finite. The document provides formulas for calculating the cardinality of this monoid, its group of units, and explores special cases of $G$-equivariant functions, known as fixing elementary collapsings. All of these results are expressed in terms of specific properties of the $G$-set, including the number of orbits and certain indices of the subgroups acting as stabilizers.

\end{abstract}

\section{Introduction}  

For any group $G$, a $G$-set is simply a set $X$ on which $G$ acts; that is, there exists a function $\cdot : G \times X \to X $ such that $e \cdot x = x $ for all $x \in X $ and $g \cdot (h \cdot x) = (gh) \cdot x $ for all $x \in X $, $g, h \in G $ In the context of semigroup theory, $G$-sets are also known as $G$-acts. A $G$-equivariant transformation of $X$, or a $G$-endomorphism of $X$, is a function $\tau : X \to X $ such that $\tau(g \cdot x) = g \cdot \tau(x) $for all $g \in G $,  $x \in X $. These maps are fundamental in the category of $G$-sets and find applications in various branches of mathematics such as equivariant topology, representation theory, and statistical inference.\\

The set of all $G$-equivariant transformations of $X$, which are functions that commute with the group action, forms a monoid under function composition. We denote this monoid as $\EndG$, and its group of units, consisting of all bijective $G$-equivariant transformations, as $\AutG$. These objects have been extensively studied in various contexts (see \cite{cita21}, \cite{cita22}, \cite{cita23}, \cite{cita24}). Many examples of objects with group actions and associated $G$-equivariant functions have been explored, such as cellular automata, which have been the motivation for much of this work (\cite{cita6},\cite{cita7},\cite{cita8},\cite{cita9}), continuous actions on topological spaces, independence algebras, among others. Further details can be found in (\cite{cita11}, \cite{cita12}, \cite{cita13}, \cite{cita14}, \cite{cita15}).\\

Let $\ConjG$ the set of all conjugacy classes of subgroups of $ G $, and denote them as follows: $ [H]:=\{g^{-1}Hg:\ g\in G\} $.  In the cases where there are available a numerable amount of conjugacy classes, denote $ [H_{1}], [H_{2}],...,[H_{r}],... $  to all conjugacy classes of subgroups of $ G $, ordered by their cardinality as:
$$  |H_{1}| \leq |H_{2}| \leq \dots \leq |H_{r}|\leq \dots. $$ 
We also denote a finite set with $r$-elements as $[r]=\{1,2,...,r\}$. We can define a partial order over $\ConjG$ given as:
$$H \leq K \iff \exists g\in G\ s.t.\ H\leq g^{-1}Kg. $$

Given the action of a group $ G $  on a set $ X $, we recall the $ G $-orbits and the stabilizer of elements in $ X $  as follows, for $ x \in X $ :
$$  Gx := \{ g \cdot x \mid g \in G \},\  G_{x} := \{ g \in G \mid g \cdot x = x \}. $$ 
After, based on the stabilizer, we define the following sets in $ X $. Given $ H \leq G $, let:
$$  \mathcal{B}_{H} := \{ x \in X \mid G_{x} = H \}, $$ 
$$  \Hbox := \{ x \in X \mid [G_{x}] = [H] \}, $$ 
and we extend the partial order of the conjugacy classes to theses sets as 
$$\Hbox \leq \Kbox \iff [H]\leq [K].$$

If an element $x\in X$ is in a box $\Hbox$, the whole orbits must be in the same box because of the $G$-equivariance, then we define an operator that expresses the amount of orbits in a box as $\alpha_{[H]}$.\\
 
 Note that some conjugacy classes may not be included in the stabilizer set of the action of $ G $  on $ X $, then we define the set of subgroups of $G$ that work as stabilizer for elements in $X$ as:
$$  \StabG:=\{G_{x}|\ x\in X\}. $$

However, we must note that if $ H \in \StabG $, the complete conjugacy class of $ H $, $ [H] $, is contained in $ \StabG $, because of the $G$-equivariance. If $ h \in G_{x} $  for some $ x \in X $, it holds that 
$$  (g^{-1}hg) \cdot (g^{-1}\cdot x) = g^{-1}\cdot x, \ \forall g \in G, $$ 
meaning that any conjugate of $ h $  stabilizes at least one element in $ X $. Then we denote by $\ConjX$ to the set of conjugacy classes of subgroups of $G$ in $\StabG$ as,
$$\ConjX:=\{[H]\in \ConjG|\ H\in \StabG\}.$$

The following result is well known in the theory of $ G $-equivariant functions. 
\begin{lema}\label{lema1}
Let $ G $  be a group acting on a set $ X $, given $ x,y \in X $, the following holds: 

\begin{enumerate} 
\item[i)] There exists a $ G $-equivariant function $ \tau \in \EndG $  such that $ \tau(x)=y $  if and only if $ G_{x} \leq G_{y} $.
\item[ii)] There exists a bijective $ G $-equivariant function $ \sigma \in \AutG $  such that $ \sigma(x)=y $  if and only if $ G_{x} = G_{y} $.
\end{enumerate}
\end{lema}
\noindent Further details on this result can be found in \cite{paper}.\\

Given  a subgroup $ H\leq G $  and a subset $ N\subseteq G$, we define the $ N $-conjugacy classes of $ H $  as:
$$  [H]_{N}:=\{n^{-1}Hn:\ n\in N\}. $$ 
It is easy to see that $ [H]_{N} \subseteq [H] $, meaning that  the elements in an $ N $-conjugacy class of $ H $ are some of the conjugate subgroups of $ H $, specifically those given by conjugating elements in $ N $. Denote the normalizer of a subgroup $ H $  simply as $ N_{G}(H)=N_{H} $.\\
We define some set for the action that will be helpful to get the desired results. 
$$\mathcal{U}(H,K):=\{[T]_{N_{H}}|\ H \leq T,\  T \sim_{G} K\}$$
$$\mathcal{U}(H):=\{[K]_{N_{H}}|\ H \leq K\}$$

Let recall the kernel of a function, a equivalence relation over $X$, that can be defined as 
$$ker(f)=\{(a,b)\in X \times X|\ f(a)=f(b)\}.$$

There exist some special $G$-equivariant functions called elementary collapsings, define as follows: a $G$-equivariant function $\tau\in \EndG$ is an elementary collapsing of type $(H,[K]_{N})$ if there exist $x,y\in X$ such that:
\begin{enumerate}
\item[i)] $Gx\neq Gy$.
\item[ii)] $G_{x}=H$
\item[iii)] $[G_{y}]_{N_{H}}=[G_{\tau(x)}]_{N_{H}}=[K]_{N_{H}}$.
\item[iv)] 
\begin{small}
$\ker(\tau)=\{(a,a): a\in X\} \cup \{(g\cdot x, g\cdot y),(g\cdot y, g\cdot x): g\in G\} \cup \{(g\cdot x, h\cdot x): h^{-1}g\in G_{y}\}$.
\end{small}
\end{enumerate}
The function $\tau$ is also called a fixing elementary collapsing if in addition to the properties mentioned above it holds that
$$Fix(\tau):=\{x\in X|\ \tau(x)=x \}=X\setminus Gx.$$

Based on the results and definition just mentioned, and given $x,y\in X$ such that $G_{x}\leq G_{y}$, we define the following non-bijective $G$-equivariant functions:
$$  [x\mapsto y](z)= \left\{  \begin{array}{cc}
    g\cdot y &  z=g\cdot x, \\
    z & \mbox{otherwise.}
\end{array} \right. $$ 

$$  (x\leftrightarrow y)(z)= \left\{  \begin{array}{cc}
    g\cdot y &  z=g\cdot x, \\
    g\cdot x &  z=g\cdot y, \\
    z & \mbox{otherwise.}
\end{array} \right. $$ 

Note that if $Gx \neq Gy$, then $[x\mapsto y]$ is neither injective nor surjective, while, if $Gx=Gy$ and $G$ is a finite group (as it will be considered in further cases), then these functions are bijective and will be denoted as $(x\mapsto y)$. Furthermore, all the fixing elementary collapsings in $\EndG$ are of the form $[x\mapsto y]$.\\

The fixing elementary collapsings are important in this monoid, because they form a generating set for the whole monoid modulo its group of units, i.e.,
$$\EndG=\langle\AutG \cup \{[x\mapsto y]|\ x,y\in X, G_{x}\leq G_{y}\}  \rangle.$$

The proof of this statement can be found in \cite{paper}.

\section{Expressions to compute the cardinality}

In this section we present expressions that allow us to compute the cardinality of several relevant sets within the monoid of $G$-equivariant functions, such as the number of types of elementary collapsings, the number of elementary collapsings, and the cardinality of the monoid itself, as well as its group of units. The proofs of the results presented in this space are primarily combinatorial exercises, which are not particularly difficult to deduce for those who are reasonably familiar with the structure of the monoid.\\

The first result, and part of the motivation of this work, has been proven in \cite{paper}, it is mentioned just as reference. 

\begin{teorema}
The number of different types of elementary collapsings is exactly 
$$\sum_{[H]\in \ConjX}{|\mathcal{U}(H)|-\kappa_{G}(X)}.$$
\end{teorema}

\begin{ejemplo}\label{ejemplouno}
Consider $\mathbb{Z}_{2}=\{\overline 0, \overline 1\}$ and the following $\mathbb{Z}_{2}$-set. \\
\begin{figure}[ht]\label{figurauno}
\centering
\includegraphics[width=2.5in]{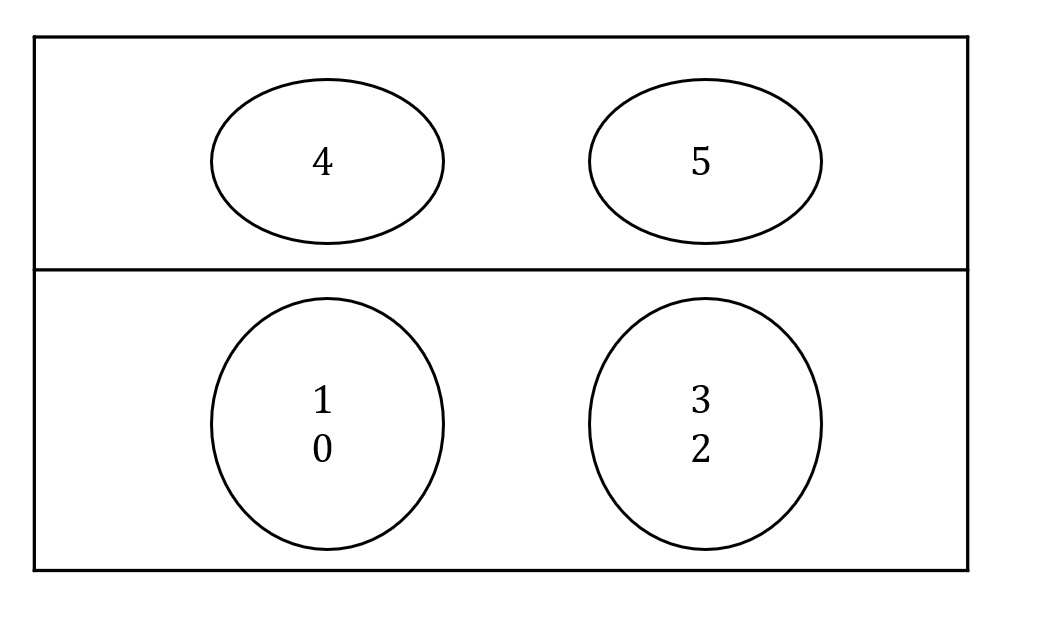}
\caption{A $\mathbb{Z}_{2}-set$.}
\end{figure}
From the figure \ref{figurauno} we can infer that
$$\begin{array}{ccc}
\overline 1\cdot 1=0,&\overline 1\cdot 2=3,& g\cdot 4=4,\ \forall g\in \mathbb{Z}_{2} \\
\overline 1\cdot 0=1,&\overline 1\cdot 3=2,& g\cdot 5=5,\ \forall g\in \mathbb{Z}_{2}  
\end{array}$$
In order to follow the properties of the $G$-equivariant functions it is enought to know where an element in an orbit goes to know where any other element in the same orbit is mapped to. This means that if we know where a functions $\tau \in \EndG$ maps the element $1$ it is also known where it maps the element $0$. For example, if it holds that $\tau(1)=2$, then it must be satisfied that $\tau(0)=3$.  
To know the total amount of $G$-equivariant functions it is needed to know where an element in each orbit can be mapped. For instance, the number $1$ can be mapped to every element in $X$, but the number $4$ can only be mapped to itself and to the number $5$.  \\
The amount of options for the number $0$, $2$ and $3$ is the same, by because of the $G$-equivariance, the functions that maps the number $1$ also map the number $0$, and they're considered the same functions, so we only have to multiply the options for the number $1$, the number $2$, the number $4$ and the number $5$. Then: 
$$|\EndG|=(6)(6)(2)(2)=144.$$

\end{ejemplo}

Hereafter, we list some basic properties of the action and the $G$-equivariant functions that will be helpful to accomplish the objectives of this paper, they can be found in any document as \cite{paper}, \cite{Howie} and \cite{rotman}. 

\begin{lema}
Given the action of a group $G$ over a finite set $X$, and an element $x\in X$ such that $G_{x}=H$, then it holds that

\begin{enumerate}
\item $$G_{g\cdot x}=g^{-1} G_{x}g=G_{x} \iff g\in N_{G}(G_{x}).$$
\item $$n_{1}\cdot x= n_{2}\cdot x \iff n_{1}H=n_{2}H.$$
\item $$|\mathcal{B}_{G_{x}}\cap Gx|=[N_{G}(G_{x}):G_{x}].$$

\item Given subgroups $H,K$, the cardinality of a $N_{H}$-conjugacy class is invariant over conjugations, i.e. 
$$|[K]_{N_{H}}|=|[gKg^{-1}]_{N_{H}}|,\ \forall g\in G.$$
\end{enumerate}
\end{lema}

We propose results that allow us to determine the cardinal of the monoid based on its structure.

\begin{teorema}
Let be $G$ a group that acts over a finite set $X$, then it holds that:
$$|\EndG|=\prod_{[H]\in \ConjX}{\left(\sum_{\substack{[K]\in \ConjX \\ [H] \leq [K]}} \alpha_{[K]}[N_{G}(K):K] |[K]_{N_{H}}|  |\mathcal{U}(H,K)| \right)^{\alpha_{[H]}}}.$$
\end{teorema}

\begin{proof}
By a combinatorial argument, to get this prove done, is enough to see how many options an element in $X$ has to be mapped, and multiply them all together, but because of the $G$-equivariance some elements in $X$ can be discounted. \\

Given an element $x\in X$, as $X$ is finite, in consequence, because of the orbit-stabilizer theorem, $H=G_{x}$ has to have finite index, therefore, $H$ is not contained in any of its conjugates, thus $|[G_{x}]_{N_{G_{x}}}|=1$, and also the amount of different $N_{H}$-conjugacy classes is $|\mathcal{U}(G_{x},G_{x})|=1$. In consequence, the amount of element in the same orbit $Gx$ such that $x$ can be mapped by a $G$-equivariant functions $\tau\in \EndG$ is given by $[N_{G}(G_{x}):G_{x}]$. Besides, the number of elements in the other orbits, but in the same box, such that $x$ can be mapped is the same, so we have to multiply this number times the number of $G$-orbits in the box, $\alpha_{[G_{x}]}$. This means that the amount of elements in the same box such that $x$ can be mapped satisfies the following expression:
$$\alpha_{[G_{x}]}[N_{G}(G_{x}):G_{x}] |[G_{x}]_{N_{G_{x}}}||\mathcal{U}(G_{x},G_{x})|.$$

Afterwards, for a different box $\Kbox$, the amount of elements in an intersection $\mathcal{B}_{K} \cap Gy$ such that $x$ can be mapped is given again for the index $[N_{G}(K):K]$, but in this case, $x$ can be mapped to another intersection given by the conjugation of $K$ by an element in $N_{H}$, but the amount of elements is the same, so we have to multiply it for the cardinality of the $N_{H}$-conjugacy classes in $\Kbox$, $|[K]_{N_{H}}|$. It is also possible that there exists another $N_{H}$-class in $\Kbox$ such that $x$ can be mapped, but the cardinality of this class is invariant by conjugation, so we only have to multiply it times the number of different $N_{H}$-classes in $\Kbox$, $|\mathcal{U}(H,K)|$ and finally multiply it times the amount of orbits in $\Kbox$. Then, the amount of elements in $\Kbox$ such that $x$ can be mapped is given by the expression:
$$\alpha_{[K]}[N_{G}(K):K] |[K]_{N_{H}}|  |\mathcal{U}(H,K)|.$$

Is not difficult to see that $|\mathcal{U}(H,K)|=0$ if $[H]\nleq [K]$. Then, to know the total amount of different elements in which is possible to map $x$ can be determined only summing this expression as long as $[H]\leq [K]$, and $[K] \in \ConjX$.

$$\sum_{\substack{[K]\in \ConjX \\ [H] \leq [K]}} \alpha_{[K]}[N_{G}(K):K] |[K]_{N_{H}}|  |\mathcal{U}(H,K)|.$$

For the complete box $\Hbox$, as only one element in each orbit is needed to determine the images of all the elements in the orbit, it's enought to multiply this quantity by itself the number of orbits in $\Hbox$ times. 

$$\left( \sum_{\substack{[K]\in \ConjX \\ [H] \leq [K]}} \alpha_{[K]}[N_{G}(K):K] |[K]_{N_{H}}|  |\mathcal{U}(H,K)|\right)^{\alpha_{[H]}}.$$

As this holds for every box $\Hbox$, we complete the calculations only multiplying this number of all the boxes by each other.  
\end{proof}

In the following, we compute the cardinality of the group of units of the monoid, this is given by all the $G$-equivariant functions that are also bijective. We begin by illustrating a particular case through an example, in order to gain intuition about the demonstration process.

\begin{ejemplo}
Consider again the $G$-set of the example \ref{ejemplouno}.\\
An ''automorphism'' $\sigma\in \AutG$ that maps $1$ into $2$, has to map $0$ exclusively to the number $3$, so we consider this as the same case. So the options for a $G$-equivariant bijective functions to map the number $1$ are the number $0$, the number $1$, the number $2$ and the number $3$. Once that an element is map to an orbit, no other element can be map to any other element in the same orbits, this would break the injectivity, because of the $G$-equivariance. So, once that we map the number $1$ for example to the number $0$, the number $2$ cannot be mapped to the same orbit, so the only options left are the number $2$ itself and the number $3$.  By the same reasons, the number $4$ can only be mapped to the number five or itself.  \\
Once again, because of the $G$-equivariance, we only consider one element in each orbit and multiply the amount of options available. In this particular case:
$$|\AutG|=(4)(2)(2)(1).$$
\end{ejemplo}

\begin{teorema}
Let be $G$ a group acting over a finite set $X$, then it holds that:
$$|\AutG| =  \prod_{[H]\in \ConjX} \left( \alpha_{[H]}\right)![N_{G}(H):H]^{\alpha_{[H]}}.$$
\end{teorema}

\begin{proof}
As an element $x\in X$ can only by mapped by a bijective $G$-equivariant function to an element in the same box, its enough to count the options in the same box. Consider $G_{x}=H$, because of the finiteness of $X$, $H$ is a subgroup of finite index, in consequence, it is not contained in any of its conjugates, and therefore, there is no $G$-equivariant function (neither regular nor bijective) that can map $x$ to an element with a different stabilizer. It is already known that the amount of elements in every orbit such that x can be mapped is given by the index $[N_{G}(H):H]$, so we multiply this quantity times the amount of orbits in $\Hbox$,
$$\alpha_{[H]}[N_{G}(H):H].$$
Once that we know where x , and all its orbit, is mapped, the next element to consider cannot be mapped to the same orbit, so we lose $[N_{G}(H):H]$ many elements as option. So for the element representing the next orbit the amount of options is 
$$\alpha_{[H]}[N_{G}(H):H]-[N_{G}(H):H]=(\alpha_{[H]}-1)[N_{G}(H):H].$$
For the third orbit the amount of options is 
$$(\alpha_{[H]}-2)[N_{G}(H):H].$$
We continue this construction until the last orbit in $\Hbox$, and the amount of possible elements to be mapped is 
$$(\alpha_{[H]}-(\alpha_{[H]}-1))[N_{G}(H):H].$$
Once again, by a combinatorial argument, we multiply all the available options for each element (and in consequence each orbit) in $\Hbox$:
\begin{small}
$$\alpha_{[H]}[N_{G}(H):H](\alpha_{[H]}-1)[N_{G}(H):H]...(\alpha_{[H]}-(\alpha_{[H]}-1))[N_{G}(H):H]=(\alpha_{[H]})[N_{G}(H):H]^{\alpha_{[H]}}.$$
\end{small}
This express only the options for one $\Hbox$, so we complete the calculations multiplying all the options in every $\Hbox$ in the action. 
\end{proof}

The elementary collapsings are important elements in $\EndG$, because they give us a generating set for $\EndG$ modulo $\AutG$. Some basic properties of the collapsings and the amount of types has been already computed in \cite{paper}, we propose an expression that counts all the different elementary collapsings without regard of the type of the collapsings. 

\begin{ejemplo}
For the $G$-set in the example \ref{ejemplouno}, there only exist 10 fixing elementary collapsings.

$$\begin{array}{|cc|}
\hline
[0\mapsto 2] & [2 \mapsto 0]\\

[0\mapsto 3] & [2 \mapsto 1]\\

[0\mapsto 4] & [2 \mapsto 4]\\

[0\mapsto 5] & [2 \mapsto 5]\\

[4\mapsto 5] & [5 \mapsto 4]\\ 
\hline
\end{array}$$

It is important to point out that $[0\mapsto 2]=[1\mapsto 3]$, because of the $G$-equivariance. 
\end{ejemplo}

\begin{teorema}
The number of different fixing elementary collapsings of any type is given by the following expression:
$$\sum_{[H]\in \ConjX}{\alpha_{[H]}\left( \sum_{\substack{[K]\in \ConjX \\ [H] \lneq [K]} }{ \alpha_{[K]}[N_{G}(K):K]|[K]_{N_{H}}| |\mathcal{U}(H,K)|+(\alpha_{[H]}-1)[N_{G}(H):H]} \right)}.$$
\end{teorema}

\begin{proof}
Once more, because of the $G$-equivariance, it's enough to count the options an element in every orbit has. An we already know that the options for an element in an orbit in the box $\Hbox$ is
$$\alpha_{[H]}[N_{G}(H):H]|[H]_{N_{H}}||\mathcal{U}(H,K)|.$$
But a fixing elementary collapsing cannot map an orbit to itself, because this would result into a bijective function, and the elementary collapsings are considered non-bijective $G$-equivariant functions. So we discard one orbit in $\Hbox$.
$$(\alpha_{[H]}-1)[N_{G}(H):H]|[H]_{N_{H}}||\mathcal{U}(H,H)|.$$

The difference lies in the fact that we're not constructing functions that maps more elements, this functions only map the elements in one orbit, so there is no need to multiply. So we sum all the options for an element in the box $\Hbox$ through the classes $[K]$ such that $[H] \leq [K]$, except in the same box $\Hbox$, because of the discarted orbit, 
$$\sum_{\substack{[K]\in \ConjX \\ [H] \lneq [K]} }{ \alpha_{[K]}[N_{G}(K):K]|[K]_{N_{H}}| |\mathcal{U}(H,K)|+(\alpha_{[H]}-1)[N_{G}(H):H]}.$$

For each orbit in $\Hbox$, there is the same amount of options, so we multiply it times the amount of orbits,   
$$\alpha_{[H]}\left( \sum_{\substack{[K]\in \ConjX \\ [H] \lneq [K]} }{ \alpha_{[K]}[N_{G}(K):K]|[K]_{N_{H}}| |\mathcal{U}(H,K)|+(\alpha_{[H]}-1)[N_{G}(H):H]} \right).$$
But this happens in all the boxes, so we sum this quantities while $[H]\in \ConjG$, and we get the result. 
\end{proof}

\end{document}